\documentclass[12pt,a4paper]{article}

\usepackage{amsmath,amsthm,amsfonts,amssymb} 
\usepackage{txfonts,eucal} 

\DeclareMathOperator\Char{Char}%
\DeclareMathOperator\Gal{Gal}%
\DeclareMathOperator\diag{diag}%
\DeclareMathOperator\GL{GL}%
\newcommand{\li}{\langle} 
\newcommand{\re}{\rangle} 
\newcommand{\lire}{\li\,\cdot\,,\cdot\,\re} 
\newcommand{\quer}{\overline{\phantom{a}}}
\newcommand{\ol}[1]{\overline{#1}}
\newcommand{\x}{\times} 

\newcommand{\rip}{\mathrel{\parallel_{r}}}
\newcommand{\lep}{\mathrel{\parallel_{\ell}}}

\newcommand{\relpi}[1]{\mathrel{\pi_{#1}}}
\newcommand{\lre}{\relpi{\ell r}}
\newcommand{\riS}{\cS_{r}}
\newcommand{\leS}{\cS_{\ell}}
\newcommand{\lrC}{\cC_{\ell r}}

\newcommand{\XX}{\phantom{-}} 

\newcommand{\bPH}{\bP(H)} 

\newcommand{\dsp}{(\bP,{\lep},{\rip})} 
\newcommand{\dspH}{\bigl(\bPH,{\lep},{\rip}\bigr)}
\newcommand{\cLH}{\cL(H)} 
\newcommand{\cAH}{\cA(H)} 

\renewcommand{\phi}{\varphi}

\newcommand{\cA}{{\mathcal A}} 
\newcommand{\cB}{{\mathcal B}}
\newcommand{\cC}{{\mathcal C}}
\newcommand{\cD}{{\mathcal D}}

\newcommand{\cF}{{\mathcal F}}

\newcommand{\cL}{{\mathcal L}}

\newcommand{\cP}{{\mathcal P}}

\newcommand{\cS}{{\mathcal S}}
\newcommand{\cR}{{\mathcal R}}

\newcommand{\bP}{{\mathbb P}}
\newcommand{\bR}{{\mathbb R}}
\newcommand{\bQ}{{\mathbb Q}}

 \newtheorem{thm}{Theorem}[section]
 \newtheorem{cor}[thm]{Corollary}
 \newtheorem{lem}[thm]{Lemma}
 \newtheorem{prop}[thm]{Proposition}
 \theoremstyle{definition}
 \newtheorem{defn}[thm]{Definition}
 \theoremstyle{remark}
 \newtheorem{rem}[thm]{Remark}
 
 \numberwithin{equation}{section}

 \newtheorem{exa}[thm]{Example}

\begin{document}

\author{Hans Havlicek \and Stefano Pasotti \and Silvia Pianta}
\title{Clifford-like parallelisms}
\date{\textit{Dedicated to Helmut Karzel on the occasion of his 90th birthday}}

\maketitle

\begin{abstract}
Given two parallelisms of a projective space we describe a construction, called
blending, that yields a (possibly new) parallelism of this space. For a
projective double space $\dsp$ over a quaternion skew field we characterise the
``Clifford-like'' parallelisms, \emph{i.e.} the blends of the Clifford
parallelisms $\lep$ and $\rip$, in a geometric and an algebraic way. Finally,
we establish necessary and sufficient conditions for the existence of
Clifford-like parallelisms that are not Clifford.
\end{abstract}

\maketitle

\section{Introduction}
The first definition of parallel lines in the real projective $3$-space dates
back to 1873 and was introduced by W.K.~Clifford in the metric framework of
elliptic geometry (see \cite{Cliff-873}): two distinct lines $M$ and $N$ in the
real elliptic $3$-space, are said to be \emph{Clifford parallel}, if the four
lines $M$, $N$, $M^\perp$ and $N^\perp$ are elements of the same regulus. (Here
$M^\perp$ denotes the polar line of $M$ w.r.t. the ``absolute'', \emph{i.e.}
the imaginary quadric that determines the elliptic metric in the real
projective $3$-space).
\par

Some years later, in 1890 F.~Klein revived Clifford's ideas and, using the
complexification of the real projective space, defined two lines to be
\emph{parallel in the sense of Clifford} if they meet the same complex
conjugate pair of generators of the absolute (see \cite{Klein-90}).

\par
Depending on the kind of generators under consideration, one can speak of
\emph{right} parallel, or \emph{left} parallel lines, then each fixed conjugate
pair of generators ``indicates'' a \emph{left} (or \emph{right}) \emph{parallel
class}, which in fact is a regular spread, namely an elliptic linear congruence
of the projective space.

\par
Thus we can say that a \emph{Clifford parallelism} in the real projective
3-space consists of all regular spreads, or elliptic linear congruences, whose
indicator lines are pairs of complex conjugate lines of a regulus contained in
an imaginary quadric. Besides, Clifford parallelisms go in pairs, and also note
that all (real) Clifford parallelisms are projectively equivalent. An
interesting survey on the various definitions of Clifford parallelisms can be
found in \cite{bett+r-12a}.

\par
Generalising this situation, H.~Karzel, H.-J.~Kroll and K.~S\"{o}rensen in 1973
introduced the notion of a \emph{projective double space} $\dsp$, that is a
projective space (of unspecified dimension, over an unspecified field) equipped
with two parallelism relations fulfilling a configurational property which can
be expressed by the axiom (DS) of Section~\ref{se:parallelisms} (see
\cite{kks-73}, \cite{kks-74}). The real projective $3$-space with left and
right Clifford parallelisms is an example and it turns out that the projective
double spaces $\dsp$ with ${\lep}\neq{\rip}$ are necessarily of dimension $3$
and precisely the ones that can be obtained from a quaternion skew field $H$
over a field $F$ as in Section~\ref{se:mixed} (see \cite{kks-73}, \cite{kks-74}, \cite{kk-75}, \cite{kroll-75}).

\par
In this way one obtains what in 2010 A.~Blunck, S.~Pasotti and S.~Pianta called
\emph{generalized Clifford parallelisms} in the note \cite{blunck+p+p-10a}. If
the maximal commutative subfields of $H$ are not mutually $F$-isomorphic, then
new ``non-Clifford'' regular parallelisms can be obtained by ``blending'' in
some suitable way the left and right parallel classes (see
\cite[4.13]{blunck+p+p-10a}). This method has no equivalent in the classical
case, since the maximal commutative subfields of the real quaternions are
mutually $\bR$-isomorphic.

\par
Taking up this idea, we introduce here the definition of \emph{Clifford-like
parallelism} in a projective double space $\dsp$, that is a parallelism
$\parallel$ on $\bP$ such that
\begin{equation*}
\forall\;M,N\in\cL \colon    M\parallel N \Rightarrow (M\lep N\mbox{~~or~~} M\rip N) .
\end{equation*}

\par
In Section \ref{se:equiv_rel} we start from the more general setting of
equivalence relations on a set $\cL$ and we define a \emph{blend} of two
equivalence relations $\relpi1$, $\relpi2$ as an equivalence relation
${\relpi3}$ such that each equivalence class of $\relpi3$ coincides with an
equivalence class of $\relpi1$ or $\relpi2$. In order to obtain a
characterisation of all the blends of $\relpi1$ and $\relpi2$ in
Proposition~\ref{prop:all_blends}, we use the equivalence relation $\relpi{12}$
generated by them, which is the join of $\relpi1$ and $\relpi2$ in the lattice
of equivalence relations on $\cL$.

\par
In Section \ref{se:parallelisms} we study the blends of parallelisms of a
projective space. By Theorem \ref{thm:par}, we can prove that the Clifford-like
parallelisms of a projective double space $(\bP,\lep,\rip)$ are precisely the
``blends'' of $\lep$ and $\rip$. Therefore Clifford-like parallelisms are
regular.

\par
In Section \ref{se:mixed} we connect a projective double space $\dsp$ with
${\lep}\neq{\rip}$ to a quaternion skew field $H$ over a field $F$ and we
describe the equivalence relation $\lre$ generated by $\lep$ and $\rip$ using
the maximal commutative subfields of $H$ (see Theorem \ref{thm:equi}). In
Theorem \ref{thm:Clike} we obtain a characterisation of all Clifford-like
parallelisms of $\dspH$ showing that they are precisely those introduced in
\cite[4.13]{blunck+p+p-10a}. Finally, in Theorems \ref{thm:ordi} and
\ref{thm:really_new} we discuss the existence and some properties of
Clifford-like parallelisms that are not Clifford.
\par
To conclude, we observe that it might be interesting to investigate the blends
of the left and right parallelisms of an arbitrary kinematic space in the same
spirit as in \cite{paso-10a}.

\section{Blends of equivalence relations}\label{se:equiv_rel}

Throughout this section we consider an arbitrary set $\cL$. Let
${\relpi{}}\subseteq\cL\times\cL$ be an equivalence relation on $\cL$. The
partition of $\cL$ associated with $\relpi{}$ is denoted by $\Pi$. The elements
of $\Pi$ are called \emph{$\relpi{}$-classes}. For any $M\in\cL$ we
denote by $\cC_{}(M)$ the $\relpi{}$-class containing $M$. The same kind of
notation will be used for other equivalence relations on $\cL$ by
writing, for example, $\relpi{1}$, $\Pi_1$ and $\cC_{1}(M)$.

\par

The following simple lemma will be used repeatedly.

\begin{lem}\label{lem:part_B}
Let $\relpi{}$ be an equivalence relation on $\cL$ and\/ $\cB\subseteq\cL$.
Then the following are equivalent.
\begin{enumerate}
\item\label{lem:part_B.a} $\cB$ admits a partition by $\relpi{}$-classes.

\item\label{lem:part_B.b} $\{\cC(X)\mid X\in \cB\}$ is the only partition
    of\/ $\cB$ by $\relpi{}$-classes.

\item\label{lem:part_B.c} $\cB=\bigcup_{X\in\cB}\cC_{}(X)$.

\item\label{lem:part_B.d} $\cL\setminus\cB$ admits a partition by
    $\relpi{}$-classes.
\end{enumerate}
\end{lem}
\begin{proof}
\eqref{lem:part_B.a} $\Rightarrow$ \eqref{lem:part_B.b}. Let $\Sigma $ be a
partition of $\cB$ by $\relpi{}$-classes. Then $\Sigma$ coincides with the
partition given in \eqref{lem:part_B.b}.
\eqref{lem:part_B.b}~$\Rightarrow$~\eqref{lem:part_B.c}. This is obvious.
\eqref{lem:part_B.c}~$\Rightarrow$~\eqref{lem:part_B.d}. It suffices to observe
that $\{\cC(X)\mid X\in \cL\setminus\cB\}$ is a partition of $\cL\setminus\cB$
by $\relpi{}$-classes. \eqref{lem:part_B.d}~$\Rightarrow$~\eqref{lem:part_B.a}.
The existence of a partition, say $\Sigma'$, of $\cL\setminus\cB$ by
$\relpi{}$-classes implies that $\Pi\setminus\Sigma'$ is a partition of $\cB$
by $\relpi{}$-classes.
\end{proof}
\par
We now introduce our basic notion.

\begin{defn}\label{def:blend}
Let $\relpi1$ and $\relpi2$ be (not necessarily distinct) equivalence relations
on $\cL$. An equivalence relation $\relpi3$ on $\cL$ is called a
\emph{blend}\/ of $\relpi1$ and $\relpi2$ if
\begin{equation}\label{eq:blend}
    \Pi_3 \subseteq \Pi_1\cup \Pi_2 .
\end{equation}
\end{defn}
Equivalently, condition \eqref{eq:blend} can be written in the form
\begin{equation}\label{eq:blend1or2}
    \forall\;M\in\cL \colon    \cC_3(M) = \cC_1(M) \mbox{~~or~~}  \cC_3(M) = \cC_2(M).
\end{equation}
The \emph{trivial blends} of $\relpi1$ and $\relpi2$ are the relations
$\relpi1$ and $\relpi2$ themselves.
\par
Our aim is to describe all blends of equivalence relations $\relpi1$ and
$\relpi2$ on $\cL$. We thereby use that all equivalence relations on
$\cL$ constitute a lattice; see, for example, \cite[Ch.~I, \S8, Ex.~9]{birk-79a} or \cite[Sect.~50]{szasz-63a}. In
this lattice, the \emph{meet} of $\relpi1$ and $\relpi2$ equals
${\relpi1}\cap{\relpi2}$ (however, the meet of equivalence relations is
irrelevant for our investigation). The \emph{join} of $\relpi1$ and $\relpi2$
is the intersection of all equivalence relations on $\cL$ that contain
$\relpi1\cup\relpi2$ or, in other words, the equivalence relation
\emph{generated} by $\relpi1$ and $\relpi2$. This join is denoted by
$\relpi{12}$. For all $M,N\in \cL$, we have $M\relpi{12} N$
precisely when there exist an integer $n\geq 1$ and (not necessarily distinct)
elements $N_1,N_2,\ldots,N_{2n+1}\in\cL$ such that
\begin{equation}\label{eq:pi_12}
    M=N_1 \relpi{1} N_{2}\relpi{2} N_3 \relpi{1}\cdots \relpi{2} N_{2n+1}=N.
\end{equation}
Also, we need another elementary lemma.

\begin{lem}\label{lem:part_12}
Let $\relpi{1}$ and $\relpi{2}$ be equivalence relations on $\cL$ and denote by
$\relpi{12}$ the equivalence relation generated by $\relpi{1}$ and $\relpi{2}$.
Furthermore, let $\cB\subseteq\cL$. Then the following statements are
equivalent.
\begin{enumerate}
\item\label{lem:part_12.a} $\cB$ admits a partition by $\pi_{12}$-classes.

\item\label{lem:part_12.b} $\cB$ admits a partition by
    $\relpi{1}$-classes and a partition by $\relpi2$-classes.
\end{enumerate}
\end{lem}

\begin{proof}
Let \eqref{lem:part_12.a} be satisfied. For each $X\in\cB$, we have
$\cC_{1}(X)\cup\cC_{2}(X)\subseteq\cC_{12}(X)\subseteq\cB$, where the second
inclusion follows by applying Lemma~\ref{lem:part_B} to an existing partition
of $\cB$ by $\pi_{12}$-classes. This forces
$\cB\subseteq\bigcup_{Y\in\cB}\cC_{1}(Y)\subseteq\cB$ and
$\cB\subseteq\bigcup_{Z\in\cB}\cC_{2}(Z)\subseteq \cB$. These two formulas in
combination with Lemma~\ref{lem:part_B} establish \eqref{lem:part_12.b}.
\par
Conversely, let us choose some $M\in\cB$. Then, for all $N\in\cC_{12}(M)$,
there is a finite sequence as in \eqref{eq:pi_12}, whence $N\in\cB$. Thus
$\cC_{12}(M)\subseteq\cB$. This shows
$\cB\subseteq\bigcup_{X\in\cB}\cC_{12}(X)\subseteq\cB$, and
Lemma~\ref{lem:part_B} provides the existence of a partition of $\cB$ by
$\relpi{12}$-classes.
\end{proof}
\begin{prop}\label{prop:all_blends}
Let $\relpi1$ and $\relpi2$ be equivalence relations on $\cL$. Furthermore,
denote by $\relpi{12}$ the equivalence relation generated by $\relpi1$ and
$\relpi2$.
\begin{enumerate}
\item\label{prop:all_blends.a} Upon choosing any subset $\cD$ of $\cL$ we
    let
    \begin{equation}\label{eq:B}
        \cB := \bigcup_{X\in \cD} \cC_{12}(X).
    \end{equation}
    Then the set
    \begin{equation}\label{eq:Pi_B}
        \Pi_\cB:= \bigl\{\cC_{1}(M)\mid M\in \cB\bigr\}
        \cup      \bigl\{\cC_{2}(M)\mid M\in \cL\setminus\cB\bigr\}
    \end{equation}
    is a partition of $\cL$, whose associated equivalence relation
    $\pi_\cB$ is a blend of $\relpi1$ and $\relpi2$.
\item\label{prop:all_blends.b} Conversely, any blend of $\relpi1$ and
    $\relpi2$ arises according to\/ \eqref{prop:all_blends.a} from at least
    one subset of $\cL$.

\item\label{prop:all_blends.c} Let $\cD$ and $\tilde\cD$ be subsets of
    $\cL$. Applying the construction from\/ \eqref{prop:all_blends.a} to
    $\cD$ and $\tilde\cD$ gives $\relpi\cB$ and $\relpi{\tilde\cB}$,
    respectively. Then $\relpi\cB$ coincides with $\relpi{\tilde\cB}$ if,
    and only if,
    \begin{equation}\label{eq:coinc}
        \{\cC_{12}(X)\mid X\in\cD\setminus\cL_{12}\} =
        \{\cC_{12}(\tilde X)\mid \tilde X\in\tilde\cD\setminus\cL_{12}\} ,
    \end{equation}
    where $\cL_{12}:=\{M\in\cL\mid \cC_{1}(M)=\cC_{2}(M)\}$.
\end{enumerate}
\end{prop}

\begin{proof}
\eqref{prop:all_blends.a} We read off from \eqref{eq:B} and
Lemma~\ref{lem:part_B} that $\cB$ admits a partition by $\relpi{12}$-classes.
Now Lemma~\ref{lem:part_12} shows that $\cB$ admits a partition by
$\relpi{1}$-classes, namely $\bigl\{\cC_{1}(M)\mid M\in \cB\bigr\}$, and also a
partition by $\relpi{2}$-classes. Applying Lemma~\ref{lem:part_B} to the
latter, gives the existence of a partition of $\cL\setminus\cB$ by
$\relpi{2}$-classes, namely $\bigl\{\cC_{2}(M)\mid M\in
\cL\setminus\cB\bigr\}$. Therefore, in accordance with \eqref{eq:Pi_B},
$\Pi_\cB\subseteq\Pi_1\cup\Pi_2$ is a partition of $\cL$, and so $\pi_\cB$ is a
blend of $\relpi1$ and $\relpi2$.
\par
\eqref{prop:all_blends.b} Given any blend $\relpi3$ of $\relpi1$ and $\relpi2$
we start by defining
\begin{equation}\label{eq:B'}
    \cD:=\bigl\{X\in\cL\mid \cC_{1}(X)=\cC_{3}(X)\bigr\}
    \mbox{~~and~~}
    \cB':=\bigcup_{X\in\cD} \cC_{1}(X).
\end{equation}
According to its definition, $\cB'$ admits a partition $\Sigma'$ by
$\relpi{1}$-classes. From \eqref{eq:B'}, $\Sigma'$ is a partition of $\cB'$ by
$\relpi3$-classes as well. Now Lemma~\ref{lem:part_B} gives that
$\cL\setminus\cB'$ admits a partition, say $\Sigma''$, by $\relpi3$-classes. No
element of $\Sigma''$ can be in $\Pi_1$. Since $\relpi3$ is a blend of
$\relpi1$ and $\relpi2$, we obtain $\Sigma''\subseteq\Pi_2$. Next, by virtue of
Lemma~\ref{lem:part_B}, $\cB'$ admits also a partition by $\relpi{2}$-classes
and, finally, Lemma~\ref{lem:part_12} provides a partition $\Sigma'''$ of
$\cB'$ by $\relpi{12}$-classes.
\par
We now proceed as in part \eqref{prop:all_blends.a} of the current proof,
commencing with the set $\cD$ given in \eqref{eq:B'}. From
Lemma~\ref{lem:part_B} and due to the existence of the partition $\Sigma'''$ of
$\cB'$, we see that the set $\cB$ from \eqref{eq:B} equals the set $\cB'$
appearing in \eqref{eq:B'}. Under these circumstances we end up with
$\Pi_\cB=\Sigma'\cup\Sigma''=\Pi_{3}$.
\par
\eqref{prop:all_blends.c} This is an immediate consequence of
\eqref{eq:Pi_B}.
\end{proof}

The set $\cD$ from Proposition~\ref{prop:all_blends}~\eqref{prop:all_blends.a}
merely serves the purpose of defining the set $\cB$ in \eqref{eq:B}. Formula
\eqref{eq:coinc} does not impose any restriction on $\cD\cap\cL_{12}$ and
$\tilde\cD\cap\cL_{12}$. Therefore, whenever $\cL_{12}$ is non-empty, there is
a choice of $\cD$ and $\tilde\cD$ such that ${\relpi\cB}={\relpi{\tilde\cB}}$
even though $\cB\neq\tilde\cB$. For example, $\cD:=\emptyset$ and
$\tilde\cD:=\cL_{12}\neq\emptyset$ give rise to
$\emptyset=\cB\neq\tilde\cB$, whereas
${\relpi\cB}={\relpi{\tilde\cB}}={\relpi2}$.

The final result in this section will lead us to a characterisation of blends
of parallelisms in Theorem~\ref{thm:par}. It is motivated by the following
evident observation. Let $\relpi1,\relpi2,\relpi3$ be equivalence
relations on $\cL$. If $\relpi3$ is a blend of $\relpi1$ and $\relpi2$ then,
by \eqref{eq:blend1or2},
\begin{equation}\label{eq:1or2}
    \forall\;M,N\in\cL \colon  M\relpi 3 N \Rightarrow (M\relpi1 N\mbox{~~or~~} M\relpi2 N) .
\end{equation}
In our current setting, \eqref{eq:1or2} is not sufficient for $\relpi3$
to be a blend of $\relpi1$ and $\relpi2$. Take, for example, as $\cL$ any set
with at least two elements, let ${\relpi1}={\relpi2}=\cL\times\cL$, and let
${\relpi3}$ be the equality relation on $\cL$. Then \eqref{eq:1or2} is
trivially true, but $\relpi3$ fails to be a blend of $\relpi1$ and $\relpi2$.

\begin{prop}\label{prop:123}
Let $\relpi1,\relpi2,\relpi3$ be equivalence relations on $\cL$ such that\/
\eqref{eq:1or2} is satisfied. Then
\begin{equation}\label{eq:subset}
    \forall\;M\in\cL \colon    \cC_3(M)\subseteq\cC_1(M) \mbox{~~or~~}  \cC_3(M)\subseteq\cC_2(M) .
\end{equation}
\end{prop}

\begin{proof}
Assume, to the contrary, that \eqref{eq:subset} does not hold. So, there is an
$M_0\in\cL$ such that $\cC_3(M_0)\not\subseteq\cC_1(M_0)$ and
$\cC_3(M_0)\not\subseteq\cC_2(M_0)$. Now $\cC_3(M_0)\not\subseteq\cC_1(M_0)$
implies the existence of an $M_2\in\cC_3(M_0)\setminus\cC_1(M_0)$. Applying
\eqref{eq:1or2} to $M_2\relpi3 M_0$ and taking into account that
$M_2\notin\cC_1(M_0)$, we obtain $M_2\in\cC_2(M_0)$. Hence
\begin{equation}\label{eq:M2}
    M_2 \in \bigl(\cC_2(M_0)\cap\cC_3(M_0)\bigr)\setminus \cC_1(M_0).
\end{equation}
Likewise, $\cC_3(M_0)\not\subseteq\cC_2(M_0)$ implies that there exists an
element
\begin{equation}\label{eq:M1}
    M_1 \in \bigl(\cC_1(M_0)\cap\cC_3(M_0)\bigr) \setminus \cC_2(M_0).
\end{equation}
All in all, $M_1\relpi3 M_0 \relpi3 M_2$ yields $M_1\relpi3 M_2$. Thus, by
\eqref{eq:1or2}, at least one of the following is satisfied.
\par
(i) $M_1\relpi1 M_2$. This gives $M_2\in\cC_1(M_1)=\cC_1(M_0)$ and contradicts
\eqref{eq:M2}.
\par
(ii) $M_1\relpi2 M_2$. This gives $M_1\in\cC_2(M_2)=\cC_2(M_0)$ and contradicts
\eqref{eq:M1}.
\end{proof}

\section{Blends of parallelisms}\label{se:parallelisms}

We consider a projective space $\bP$ with point set $\cP$ and line set $\cL$.
An equivalence relation on $\cL$ is called a \emph{parallelism} on $\bP$ if
each point $q\in\cP$ is incident with precisely one line from each equivalence
class; see, for example, \cite{john-03a}, \cite{john-10a}, or
\cite[\S~14]{karz+k-88}. The notation from the previous section will slightly
be altered when dealing with parallelisms by writing $\parallel$ instead of
$\relpi{}$. In addition, if ${\parallel}\subseteq\cL\times\cL$ is a parallelism,
then the equivalence class of a line $M\in\cL$ will be called its
\emph{parallel class}, and it will be denoted by $\cS(M)$ in order to emphasise
the fact that $\cS(M)$ is a \emph{spread} of $\bP$. On the other hand, the
partition of $\cL$ arising from $\parallel$ will be written as $\Pi$ like
before. In the presence of several parallelisms we shall distinguish
between these objects by adding appropriate indices or attributes.

As anticipated, the next theorem provides a characterisation of blends of
parallelisms by virtue of Proposition~\ref{prop:123}.

\begin{thm}\label{thm:par}
Let\/ $\parallel_1$ and\/ $\parallel_2$ be parallelisms on\/ $\bP$. Then the
following hold.
\begin{enumerate}
\item\label{prop:par.a} Any blend of\/ $\parallel_1$ and\/ $\parallel_2$ is
a parallelism on\/ $\bP$.

\item\label{prop:par.b} A parallelism\/ $\parallel_3$ on\/ $\bP$ is a blend
    of\/ $\parallel_1$ and\/ $\parallel_2$ if, and only if,
\begin{equation}\label{eq:par1or2}
\forall\;M,N\in\cL \colon    M\parallel_3 N \Rightarrow (M\parallel_1 N\mbox{~~or~~} M\parallel_2 N) .
\end{equation}
\end{enumerate}
\end{thm}

\begin{proof}
\eqref{prop:par.a} All parallel classes of the given parallelisms are spreads
of $\bP$. The same applies therefore to all equivalence classes of any blend of
$\parallel_1$ and $\parallel_2$, that is, such a blend is a parallelism on
$\bP$.
\par
\eqref{prop:par.b} If $\parallel_3$ is a blend of $\parallel_1$ and
$\parallel_2$ then \eqref{eq:par1or2} is nothing but a reformulation of
\eqref{eq:1or2}. Conversely, we first make use of Proposition~\ref{prop:123},
which gives \eqref{eq:subset} up to some notational differences. Next, we
notice that no proper subset of a spread of $\bP$ is again a spread of $\bP$.
Since all parallel classes of $\parallel_1$, $\parallel_2$, and $\parallel_3$
are spreads of $\bP$, we are therefore in a position to infer from
\eqref{eq:subset} that, \emph{mutatis mutandis}, \eqref{eq:blend1or2} is
satisfied.
\end{proof}

Suppose that a projective space $\bP$ is endowed with parallelisms $\lep$
and $\rip$ that are called the \emph{left} and \emph{right} parallelism,
respectively. We speak of left (right) parallel lines and left (right) parallel
classes. According to \cite{kks-73}, $\dsp$ constitutes a \emph{projective double
space} if the following axiom is satisfied.

\begin{enumerate}
  \item[(DS)] For all triangles $p_0,p_1,p_2$ in $\bP$ there exists a
      common point of the lines $M_1$ and $M_2$ that are defined as
      follows. $M_1$ is the line through $p_2$ that is left parallel to the
      join of $p_0$ and $p_1$, $M_2$ is the line through $p_1$ that is
      right parallel to the join of $p_0$ and $p_2$.
\end{enumerate}
In case of a projective double space $\dsp$, each of $\lep$ and $\rip$ is
referred to as a \emph{Clifford parallelism} of $\dsp$. We now generalise this
notion.

\begin{defn}\label{def:clifflike}
Let $\dsp$ be a projective double space. A \emph{Clifford-like} parallelism
$\parallel$ of $\dsp$ is a parallelism on $\bP$ such that
\begin{equation*}
\forall\;M,N\in\cL \colon    M\parallel N \Rightarrow (M\lep N\mbox{~~or~~} M\rip N) .
\end{equation*}
\end{defn}
By Theorem~\ref{thm:par}, the Clifford-like parallelisms of $\dsp$ are
precisely the blends of $\lep$ and $\rip$. In particular, $\lep$ and $\rip$
themselves are the trivial examples of Clifford-like parallelisms of $\dsp$.
\par
Next, we recall that there exist projective double spaces $\dsp$ such that
$\lep$ coincides with $\rip$. See \cite{havl-15}, \cite{karzel-65} and \cite{kroll-75} for further
details, an algebraic characterisation, and geometric properties. Such a double
space has only one Clifford-like parallelism, namely ${\lep}={\rip}$. We
therefore exclude this kind of double space from our further discussion.

The projective double spaces $\dsp$ with ${\lep}\neq{\rip}$
are precisely the ones that can be obtained from quaternion skew fields (see
\cite{kks-73}, \cite{kks-74}, \cite{kk-75}, \cite{kroll-75}). A detailed account is the topic of the next
section.

Finally, we observe that the ``left and right Clifford parallelisms'' introduced in \cite{blunck+k+s+s-17z} and defined by an \emph{octonion division algebra} do not give rise to a projective double
space. For further details, see \cite{blunck+k+s+s-17z} and the references therein.

\begin{rem}
In \cite[Rem.~3.7 and Thm.~3.8]{havl+r-17a} the authors gave examples of
\emph{piecewise Clifford parallelisms} with two pieces. Without going into
details, let us point out that (in our terminology) these parallelisms arise
from a three-dimensional Pappian projective space $\bP$ that is made into a
projective double space in two different ways, say
$(\bP,\parallel_{\ell,1},\parallel_{r,1})$ and
$(\bP,\parallel_{\ell,2},\parallel_{r,2})$. Thereby, it has to be assumed that
$\parallel_{\ell,1}$ and $\parallel_{\ell,2}$ share a \emph{single} parallel
class. The piecewise Clifford parallelisms with two pieces are blends of
$\parallel_{\ell,1}$ and $\parallel_{\ell,2}$, but none of these is
Clifford-like with respect to any double space structure on $\bP$. The proof of
the last statement is beyond the scope of this article, since the methods
utilised in \cite{havl+r-17a} are totally different from ours.
\end{rem}

\section{Clifford-like parallelisms from quaternion skew fields}\label{se:mixed}

In this section we deal with a quaternion skew field $H$ with centre $F$. We
thereby stick to the terminology and notation from \cite{blunck+p+p-10a} and
\cite{havl-16a}. Also, we use the abbreviations $H^*:=H\setminus\{0\}$ and
$F^*:=F\setminus\{0\}$. For a
detailed account on quaternions we refer, among others, to
\cite[pp.~46--48]{tits+w-02a} and \cite[Ch.~I]{vigneras}.
\par
The $F$-vector space $H$ is equipped with a quadratic form $H\to F$, called the
\emph{norm form}, sending $q\mapsto q\ol{q}=\ol{q}q$. Here $\quer$ denotes the
\emph{conjugation}, which is an antiautomorphism of the skew field $H$. The
conjugation is of order two and fixes $F$ elementwise. Polarisation of the norm
form yields the symmetric bilinear form
\begin{equation}\label{eq:polar}
    \lire \colon H\x H \to F \colon (p,q)\mapsto
    \li p,q\re := p\ol q + q\ol p = \ol{p}q + \ol{q}p .
\end{equation}
For any subset $X\subseteq H$ we denote by $X^\perp$ the set of those
quaternions that are orthogonal to all elements of $X$ with respect to $\lire$.
\par
The \emph{projective space} $\bP(H)$ is understood to be the set of all
subspaces of the $F$-vector space $H$ and \emph{incidence} is symmetrised
inclusion. We adopt the usual geometric language: \emph{points}, \emph{lines}
and \emph{planes} are the subspaces of vector dimension one, two, and three,
respectively. The set of lines of $\bPH$ will be
written as $\cLH$. Furthermore, we shall regard
$\perp$ as a \emph{polarity} of $\bP(H)$ sending, for example, any line $M$ to
its polar line $M^\perp$. For one kind of line this will now be made more
explicit.

\begin{lem}
For any line $L=F1\oplus Fg$, where $1\in H$ and $g\in H\setminus F$,
the line $L^\perp$ is the set of all $u\in H$ subject to the condition
\begin{equation}\label{eq:L_1-perp}
   u\ol{g} = gu .
\end{equation}
\end{lem}
\begin{proof}
From $L^\perp = {1^\perp}\cap {g^\perp}$ and \eqref{eq:polar}, a
quaternion $u\in H$ belongs to $L^\perp$ precisely when the following
{system of equations} is satisfied:
\begin{equation}\label{eq:1-g-perp}
            u+\ol{u} = 0, \quad  g\ol{u}+u\ol{g} = 0 .
\end{equation}
It is immediate from \eqref{eq:1-g-perp} that any $u\in L^\perp$ satisfies
\eqref{eq:L_1-perp}. Conversely, if \eqref{eq:L_1-perp} holds for some $u\in H$
then $g(\ol{u}+u) = g\ol{u}+gu = g\ol{u}+u\ol{g}\in F$. Together with $g\notin
F$ and $\ol{u}+u\in F$ this forces $\ol{u}+u=0$, whence the system
\eqref{eq:1-g-perp} is satisfied.
\end{proof}

Let $M,N\in\cLH$. Then the line $M$ is
\emph{left parallel} to the line $N$, in symbols $M\lep N$, if there is a $c\in
H^*$ with $cM=N$. Similarly, $M$ is \emph{right parallel} to $N$, in symbols
$M\rip N$, if there is a $d\in H^*$ with $Md=N$. The relations $\lep$ and
$\rip$ make $\bPH$ into a projective double space $\dspH$, that is, $\lep$ and
$\rip$ are its \emph{Clifford parallelisms} (see \cite{kk-75}). In accordance
with the terminology and notation from Section~\ref{se:parallelisms}, each line
$M\in\cLH$ determines its \emph{left parallel class} $\leS(M)$ and its
\emph{right parallel class} $\riS(M)$. All left (right) parallel classes are
regular spreads of $\bP$ (see \cite[4.8~Cor.]{blunck+p+p-10a} or
\cite[Prop.~4.3]{havl-16a}), that is, $\lep$ and $\rip$ are \emph{regular}
parallelisms \cite[Ch.~26]{john-10a}.


For any choice of $c,d\in H^*$ we can define the $F$-linear bijection
$\mu_{c,d}\colon H\to H\colon p\mapsto cpd$, which acts as a projective
collineation on $\bPH$ preserving both the left and the right Clifford
parallelism as a straightforward computation shows. Also, $\mu_{c,d}$ preserves
the norm form of $H$ up to the factor $c\ol{c}d\ol{d}\in F^*$ so that
orthogonality of subspaces of $H$ is preserved too. Two particular cases
deserve special mention. For $d\in F^*$, in particular for $d=1$, the mapping
$\mu_{c,d}$ is a \emph{left translation}. A \emph{right translation} arises in
a similar way for $c\in F^*$.

Let $\cAH$ be the \emph{star of lines} with
centre $F1$ (with $1\in H$), that is, the set of all lines of $\cLH$ passing
through the point $F1$. From an algebraic point of view, each left (right)
parallel class has a distinguished representative, namely its only line
belonging to $\cAH$. The star $\cAH$ is precisely the set of all
two-dimensional $F$-subalgebras of $H$ or, in other words, the set of all
maximal subfields of $H$. Given $L_1,L_2\in\cAH$
we remind that an \emph{$F$-isomorphism} of $L_1$ onto $L_2$ is a ring
isomorphism $L_1\to L_2$ fixing $F$ elementwise. If such an isomorphism exists then $L_1$ and $L_2$ are called
\emph{$F$-isomorphic}, in symbols $L_1\cong L_2$.

Let $\lre$ denote the equivalence relation on $\cLH$ that is generated by
the left and right Clifford parallelism on $\bPH$. If $M\lre N$ applies, then we
say that $M$ is \emph{left-right equivalent} to $N$.

We now present several characterisations of left-right equivalent lines.

\begin{thm}\label{thm:equi}
Let $M_1,M_2\in\cLH$ and let $L_1$ and $L_2$ be the uniquely determined lines
through the point $F1$ such that $L_1\lep M_1$ and $L_2\rip M_2$. Then the
following are equivalent.
\begin{enumerate}

\item\label{thm:equi.a} $M_1\lre M_2$.

\item\label{thm:equi.e} There exist $e_1,e_2\in H^*$ with $e_1M_1=
    M_2e_2$.

\item\label{thm:equi.b} $\leS(M_1)\cap\riS(M_2)\neq \emptyset$.

\item\label{thm:equi.d} There exists a line $M\in\cLH$ such that
    $\leS(M_1)\cap\riS(M_2) = \{M,M^\perp\}$.

\item\label{thm:equi.c} $L_1\cong L_2$.

\item\label{thm:equi.f} There exists an $e\in H^*$ with $ e^{-1} L_1
    e = L_2$.
\end{enumerate}
\end{thm}
\begin{proof}
\eqref{thm:equi.a}~$\Rightarrow$~\eqref{thm:equi.e}. By the definition of
$\lep$ and $\rip$ and by virtue of \eqref{eq:pi_12}, we obtain that $M_1\lre
M_2$ implies the existence of an integer $n\geq 1$ and elements
$g_1,g_2,\ldots, g_{2n}$ such that
    \begin{equation*}
    M_1 \lep g_1 M_1 \rip g_1 M_1g_2\lep\cdots
    \rip g_{2n-1}g_{2n-3}\cdots  g_1 M_1 g_2g_4\cdots g_{2n} = M_2
\end{equation*}
With $e_1:=g_{2n-1}g_{2n-3}\cdots g_1$ and $e_2:=(g_2g_4\cdots g_{2n})^{-1}$
the assertion follows.

\eqref{thm:equi.e}~$\Rightarrow$~\eqref{thm:equi.b}. Clearly, $e_1M_1=M_2e_2\in
\leS(M_1)\cap\riS(M_2)$.

\par

\eqref{thm:equi.b}~$\Rightarrow$~\eqref{thm:equi.d}. By our assumption, there
exists a line $M$, say, belonging to $\leS(M_1)\cap\riS(M_2)$. Also, there is a
left translation $\mu_{c_1,1}$ taking $M$ to $L_1$, \emph{i.e.}, $c_1M=L_1$.
Since $\mu_{c_1,1}$ preserves not only the left and right Clifford parallelism
but also the orthogonality of lines in both directions, it suffices to verify
that
\begin{equation*}
    \leS(L_1)\cap\riS(L_1)=\{L_1,L_1^\perp\} .
\end{equation*}
To this end we pick a quaternion $g\in L_1\setminus F$, which is maintained
throughout this part of the proof.
\par
First, we take any line $N\in \leS(L_1)\cap\riS(L_1)$. For all $u\in N^*$ we
obtain from $1\in L_1$ that $N=uL_1=L_1u$. Thus the inner automorphism
$\mu_{u^{-1},u}$ of $H$ restricts to an automorphism of $L_1$. There are two
possibilities.

\emph{Case}~(i). $\mu_{u^{-1},u}$ fixes $L_1$ elementwise. Consequently, $u$
commutes with all elements of $L_1$ or, equivalently, $u\in L_1$. Therefore $N=L_1$.

\emph{Case}~(ii). $\mu_{u^{-1},u}$ fixes $F$ elementwise, but not $L_1$. Due to
$[L_1:F]=2$, the identity on $L_1$ and the restriction of $\mu_{u^{-1},u}$ to
$L_1$ are all the elements of the Galois group $\Gal(L_1/F)$. The restriction
of the conjugation $\quer$ to $L_1$ belongs also to $\Gal(L_1/F)$. We proceed
by showing that $g\neq\ol{g}$. If $\Char F\neq 2$ then this immediate from
$g\in L\setminus F$. If $\Char F=2$ then $g\in L\setminus F$ implies $g\neq
u^{-1}gu$. Since $g$ and $u^{-1}gu$ are distinct zeros in $L_1$ of the minimal
polynomial of $g$ over $F$, which reads $X^2 + (g+\ol{g})X+g\ol{g}$, the
coefficient $g+\ol{g}$ in this polynomial does not vanish. This implies
$g=-g\neq \ol{g}$. Irrespective of $\Char F$ we therefore have that
$\mu_{u^{-1},u}$ and $\quer$ restrict to the same automorphism of $L_1$. In
particular, $u^{-1}gu=\ol{g}$, that is, $u\in L_1^\perp$ by
\eqref{eq:L_1-perp}. Therefore $N=L_1^\perp$.
\par
Finally, it remains to establish that $L_1^\perp \in \leS(L_1)\cap\riS(L_1)$.
There exists a non-zero $h\in L_1^\perp$, whence $h\ol{g}=gh$ holds according
to \eqref{eq:L_1-perp}. Due to $g\in L_1$ this yields, for all $v\in L_1$, on
the one hand $(hv)\ol{g}=g(hv)$ and, on the other hand, $(vh)\ol{g}=g(vh)$. As
$L_1^\perp$ is characterised by \eqref{eq:L_1-perp}, we obtain $hL_1\subseteq
L_1^\perp$ and $L_1 h\subseteq L_1^\perp$. Thus $L_1^\perp = hL_1 = L_1h$, as
required. %

\par
\eqref{thm:equi.d}~$\Rightarrow$~\eqref{thm:equi.c}. From $ M=c_1^{-1}L_1
\in\leS(M_1)\cap\riS(M_2)=\leS(L_1)\cap\riS(L_2)$ there is a $c_2\in H^*$ such
that $M= L_2 c_2^{-1}$. Therefore $L_1=c_1 L_2 c_2^{-1}$, and $1\in L_2$ gives
$c_1c_2^{-1}\in L_1^*$. We read off from $1\in L_1$ that $L_1$ is a subalgebra
of $H$, and so $L_1=L_1c_1c_2^{-1}$. Summing up, we have
\begin{equation*}
    L_2= c_1^{-1} L_1 c_2 = c_1^{-1} (L_1c_1c_2^{-1}) c_2=c_1^{-1}L_1 c_1 .
\end{equation*}
This allows us to define a mapping $L_1\to L_2\colon x\mapsto c_1^{-1}xc_1$,
which is an $F$-isomorphism.

\eqref{thm:equi.c}~$\Rightarrow$~\eqref{thm:equi.f}. By the Skolem-Noether
theorem (see \cite[Thm.~4.9]{jac-89}), any $F$-isomorphism $L_1\to L_2$ can be
extended to an inner automorphism of $H$. So there is an $e\in H^*$ with
$L_2=e^{-1}L_1 e$.
\par
\eqref{thm:equi.f}~$\Rightarrow$~\eqref{thm:equi.a}. By our assumptions, there
exist $d_1,d_2,e\in H^*$ with $d_1M_1=L_1$, $M_2d_2=L_2$, and $e^{-1}L_1e =
L_2$. This implies $e^{-1} d_1 M_1 e d_2^{-1} = M_2$. Thus
\begin{equation*}
    M_1 \lep e^{-1}d_1M_1 \rip e^{-1}d_1M_1ed_2^{-1} = M_2,
\end{equation*}
and \eqref{eq:pi_12} gives $M_1\lre M_2$.
\end{proof}

\begin{cor}\label{cor:le-neq-ri}
For all lines $M_1,M_2\in\cLH$ the left parallel class $\leS(M_1)$ is different
from the right parallel class $\riS(M_2)$.
\end{cor}
\begin{proof}
As $H$ is infinite, so are $\leS(M_1)$ and $\riS(M_2)$. By
Theorem~\ref{thm:equi}, $\leS(M_1)$ and $\riS(M_2)$ have at most two lines in
common, whence they cannot coincide.
\end{proof}

\begin{cor}\label{cor:le-ri-meet} Let $N\in\cLH$. Then
$\leS(N)\cap\riS(N)=\{N,N^\perp\}$.
\end{cor}
\begin{proof}
We consider Theorem~\ref{thm:equi} for $M_1=M_2=N$. Then \eqref{thm:equi.b} holds
and, by $N\in \leS(N)\cap\riS(N)$, the assertion follows from \eqref{thm:equi.d}.
\end{proof}

Note that the result from the
previous corollary is established also in \cite[(2.6)]{kk-75} but using methods
different from ours.

\begin{cor}\label{cor:galois}
Let $L$ be a maximal subfield of $H$, that is, $L$ is a line through the point
$F1$. The field extension $L/F$
is separable if, and only if, the parallel classes $\leS(L)$ and $\riS(L)$ have
two distinct lines in common.
\end{cor}
\begin{proof}
From Corollary~\ref{cor:le-ri-meet}, $\leS(L)\cap\riS(L)=\{L,L^\perp\}$.
We consider Theorem~\ref{thm:equi} for $M_1=M_2:=L$, and so $L_1=L$. Then
\eqref{thm:equi.b} holds, and we can repeat the proof of
\eqref{thm:equi.b}~$\Rightarrow$~\eqref{thm:equi.d} with $M:=L$ and $g\in
L\setminus F$. If $L/F$
is separable then $g\neq \ol{g}$ and $L\neq L^\perp$. Otherwise $g=\ol{g}$ and
$L=L^\perp$.
\end{proof}

\begin{cor}\label{cor:L_fixed} Let $L$ be a maximal subfield of $H$ and
let $u\in H^*$. Then $u^{-1}L u=L$ is equivalent to $u\in (L\cup
L^\perp)\setminus\{0\}$.
\end{cor}
\begin{proof}
If $u\in L\setminus\{0\}$ then $u^{-1}L u=L$ is obviously true. If $u\in
L^\perp\setminus\{0\}$ then \eqref{eq:L_1-perp} implies $u^{-1}L u=L$. The
converse follows from cases (i) and (ii) in the proof of
Theorem~\ref{thm:equi}, \eqref{thm:equi.b}~$\Rightarrow$~\eqref{thm:equi.d}.
\end{proof}
By the theorem of Cartan-Brauer-Hua \cite[(13.17)]{lam-01a}, for each maximal
subfield $L$ of $H$ there exists a $c\in H^*$ with $c^{-1}Lc\neq L$.
Corollary~\ref{cor:L_fixed} shows how all such elements $c$ can be found.

\begin{cor}\label{cor:F-iso-lre}
Let $L_1$ and $L_2$ be maximal subfields of $H$. Then $L_1\lre L_2$ is
equivalent to $L_1\cong L_2$.
\end{cor}

All maps $\mu_{c,d}$, with $c,d$ varying in $H^*$, constitute a subgroup of the
general linear group $\GL(H)$. This subgroup acts on the line set $\cLH$
in a natural way, thereby splitting $\cLH$ into line orbits. From
Theorem~\ref{thm:equi}, these line orbits are precisely the classes of
left-right equivalent lines. The next result gives another
interpretation in terms of \emph{flags}, that is, incident point-line
pairs.
\begin{prop}\label{prop:ellmotions}
Let $(Fp_1,M_1)$ and $(Fp_2,M_2)$ be flags of the
$3$-dimensional projective space $\bPH$. There exists a map $\mu_{c,d}$, with
$c,d\in H^*$, taking $(Fp_1,M_1)$ to $(Fp_2,M_2)$ if, and only if, $M_1\lre
M_2$.
\end{prop}
\begin{proof}
First, note that there always exists the left translation
$\mu_{p_1^{-1},1}$ taking $Fp_1$ to $F1$ and the right translation
$\mu_{1,p_2^{-1}}$ taking $Fp_2$ to $F1$. So, $L_1:=p_1^{-1}M_1 $ and
$L_2:=M_2p_2^{-1}$ are the uniquely determined lines appearing in
Theorem~\ref{thm:equi}.
\par
From Theorem~\ref{thm:equi}, $M_1\lre M_2$ implies $e^{-1}L_1e=L_2$ for some
$e\in H^*$. Letting $c:=e^{-1}p_1^{-1}$ and $d:=ep_2$ gives $cp_1d=p_2$ and
$cM_1d = M_2$, that is, the map $\mu_{c,d}$ has the required properties.
\par
Conversely, if $\mu_{c,d}$ takes $(Fp_1,M_1)$ to $(Fp_2,M_2)$ then
$cM_1=M_2d^{-1}$ forces $M_1\lre M_2$ according to Theorem~\ref{thm:equi}.
\end{proof}

\begin{rem}
The group $\Gamma$ of all collineations of $\bPH$ that preserve both the left
and the right parallelism was described in \cite[Thm.~1]{pian-87b} in terms of
the factor group $H^*/F^*$, which thereby serves as a model for the point set
$\bPH$ by identifying $F^*c$ with $Fc$ for all $c\in H^*$. By \cite[Prop.~4.1
and Prop.~4.2]{pz90-coll}, a collineation $\gamma$ of $\bPH$ belongs to
$\Gamma$ if, and only if, $\gamma$ can be induced by an $F$-semilinear
transformation of $H$ that is the product of an automorphism of the skew field
$H$ and a map $\mu_{c,d}$ for some $c,d\in H^*$ (see also \cite[Thm. 4.3]{blunck+k+s+s-17z}).

In particular, the maps $\mu_{c,d}$ induce exactly the $F$-linear part of the
group $\Gamma$. If $\Char{F}\neq2$ then we know by \cite[(4.16)]{karz+p+s-93a}
that they induce precisely the proper motions of the elliptic $3$-space
$\bPH$, so the classes of left-right equivalent lines turn out to be the line
orbits under the action of the elliptic proper motion group.
\end{rem}

%

The following result describes \emph{all} Clifford-like parallelisms of
$\dspH$.

\begin{thm}\label{thm:Clike}
Let $\cAH$ be the subset of all lines of\/ $\bPH$ through the point $F1$.
\begin{enumerate}
\item\label{thm:Clike.a} Upon choosing any subset $\cD$ of $\cAH$ we let
\begin{equation}\label{eq:F}
    \cF:= \bigcup_{X\in\cD,\,c\in H^*} c^{-1} X c .
\end{equation}
Then
\begin{equation}\label{eq:Pi_F}
    \Pi_\cF:=\bigl\{\leS(L)\mid L\in \cF\bigr\}\cup \bigl\{\riS(L)\mid L\in \cAH\setminus\cF \bigr\}
\end{equation}
is the set of parallel classes of a Clifford-like parallelism
$\parallel_\cF$, say, of the projective double space $\dspH$.

\item\label{thm:Clike.b} Conversely, any Clifford-like parallelism\/
    $\parallel$ of\/ $\dspH$ arises according to \eqref{thm:Clike.a} from
    at least one subset of $\cAH$.
\item\label{thm:Clike.c} Let $\cD$ and $\tilde\cD$ be subsets of
    $\cAH$. Applying the construction from\/ \eqref{thm:Clike.a} to
    $\cD$ and $\tilde\cD$ gives parallelisms\/ $\parallel_\cF$ and
    $\parallel_{\tilde\cF}$, respectively. Then\/ $\parallel_\cF$ coincides
    with\/ $\parallel_{\tilde\cF}$ if, and only if, $\cF=\tilde\cF$.
\end{enumerate}
\end{thm}

\begin{proof}
\eqref{thm:Clike.a} We apply the construction from
Proposition~\ref{prop:all_blends}~\eqref{prop:all_blends.a} to $\cD$; thereby
we replace $\relpi{1}$ and $\relpi{2}$ with $\lep$ and $\rip$, respectively.
So, starting with $\cB:=\bigcup_{X\in\cD} \lrC(X)$, we finally arrive at the
partition $\Pi_\cB$ from \eqref{eq:Pi_B}, whose associated equivalence relation
on $\cLH$ is a blend of $\lep$ and $\rip$. By Theorem~\ref{thm:par}, this
$\Pi_\cB$ is the set of parallel classes of a Clifford-like parallelism of
$\dspH$. Each of its parallel classes has a unique line in common with $\cAH$.
Therefore
\begin{equation}\label{eq:Pi_B_new}
        \Pi_\cB=  \bigl\{\leS(L)\mid L\in \cB\cap\cAH\bigr\}
        \cup      \bigl\{\riS(L)\mid L\in \cAH\setminus\cB\bigr\}.
    \end{equation}
Theorem~\ref{thm:equi} and Corollary~\ref{cor:F-iso-lre} show that
\begin{equation*}
    \forall\; X\in \cAH \colon \{c^{-1}Xc \mid c\in H^*\} = \{Y\in\cAH\mid X\cong Y\} =  \lrC(X)\cap\cAH .
\end{equation*}
So, substituting in \eqref{eq:F} gives
\begin{equation}\label{eq:F_new}
    \cF = \bigcup_{X\in\cD} \Bigl(\lrC(X)\cap\cAH\Bigr) =\cB\cap\cAH .
\end{equation}
Now, by comparing \eqref{eq:Pi_F} with \eqref{eq:Pi_B_new}, we obtain
$\Pi_\cF=\Pi_\cB$.

\par
\eqref{thm:Clike.b} The given parallelism $\parallel$ is a blend of $\lep$ and
$\rip$ by Theorem~\ref{thm:par}. Thus $\parallel$ allows a construction as
described in Proposition~\ref{prop:all_blends}~\eqref{prop:all_blends.a} using
$\lep$, $\rip$, and some subset, say $\tilde\cD$, of $\cLH$. Replacing
$\tilde\cD$ with the set
\begin{equation}\label{eq:D}
    \cD:= \Bigl( \bigcup_{X\in\widetilde{\cD}}\lrC(X) \Bigr) \cap \cAH
\end{equation}
does not alter this result, as has been pointed out in
Proposition~\ref{prop:all_blends}~\eqref{prop:all_blends.c}. The first part of
the current proof shows that we also get the parallelism $\parallel$ by
applying the construction from \eqref{thm:Clike.a} to the set $\cD$ from
\eqref{eq:D}.
\par
\eqref{thm:Clike.c} By the first part of the proof, we obtain $\parallel_\cF$
and $\parallel_{\tilde\cF}$ from $\cD$ and $\tilde\cD$, respectively, also via
the construction in
Proposition~\ref{prop:all_blends}~\eqref{prop:all_blends.a}. In our current
setting the condition \eqref{eq:coinc} simplifies to
\begin{equation}\label{eq:lr_coinc}
    \{\cC_{\ell r}(X) \mid X\in \cD \} = \{\cC_{\ell r}(\tilde X) \mid \tilde X\in \tilde \cD \},
\end{equation}
since $\cLH_{\ell r}=\emptyset$ by Corollary~\ref{cor:le-neq-ri}. From
\eqref{eq:F_new}, equation \eqref{eq:lr_coinc} is equivalent to
$\cF=\tilde\cF$. It therefore suffices to make use of
Proposition~\ref{prop:all_blends}~\eqref{prop:all_blends.c}, with
\eqref{eq:coinc} to be replaced by $\cF=\tilde\cF$, in order to complete the
proof.
\end{proof}

\begin{rem} Theorem~\ref{thm:Clike}~\eqref{thm:Clike.a} was sketched without
a strict proof in \cite[4.13]{blunck+p+p-10a}. However, there are some formal
differences to our approach, as we avoid the indicator lines of regular spreads
that have been used there. Our set of lines $\cAH$ is, from an algebraic point
of view, the family of all quadratic extensions $L$ of $F$ with $F\subseteq L
\subseteq H$ from \cite[4.13]{blunck+p+p-10a}. In this way, our $\cF$ turns
into a family of subfields of $H$. Equation \eqref{eq:F} guarantees that
no subfield in $\cF$ is $F$-isomorphic to a subfield in
$\cAH\setminus\cF$. The latter condition is mentioned in the sketch of proof
from \cite[4.13]{blunck+p+p-10a}, but is missing there at that point,
where the family $\cF$ is fixed for the first time. ($F$-isomorphic
subfields of $H$ are termed as being ``conjugate'' in \cite{blunck+p+p-10a}.)
\end{rem}

Below we shall make use of the \emph{ordinary quaternion algebra} over a
\emph{formally real} field $F$, \emph{i.e.} $-1$ is not a square in $F$. This
kind of algebra will be denoted as $(K/F,-1)$. According to
\cite[pp.~46--48]{tits+w-02a} it arises (up to $F$-isomorphism) in the
following way. The field $F$ is extended to $K:=F(i)$, where $i$ is a square
root of $-1\in F$. One defines $(K/F,-1)$ as the subring of the ring of
$2\x 2$ matrices over $K$ consisting of all matrices
\begin{equation*}
    \begin{pmatrix}
    \XX x_1+i x_2 & y_1+iy_2\\
    -y_1+iy_2 & x_1-ix_2
    \end{pmatrix} \mbox{~~with~~}x_1,x_2,y_1,y_2\in F
\end{equation*}
and identifies any $x\in F$ with the matrix $\diag(x,x)\in (K/F,-1)$.
The $F$-algebra $(K/F,-1)$ is a skew field if, and only if, $-1$ is not a sum
of two squares in $F$. If the latter condition applies then $(K/F,-1)$ is
called the \emph{ordinary quaternion skew field} over $F$. For example, an
ordinary quaternion skew field exists over any formally real Pythagorean field.
We recall that a field is \emph{Pythagorean} when the sum of any two squares
is a square as well (see \emph{e.g.} \cite[p. 204]{ksw-73}).

\begin{thm}\label{thm:ordi}
Let $H$ be a quaternion skew field with centre $F$. Then the following are
equivalent.
\begin{enumerate}
\item\label{thm:ordi.a} $F$ is a formally real Pythagorean field, and $H$
    is the ordinary quaternion skew field over $F$.
\item\label{thm:ordi.b} All maximal subfields of $H$ are mutually
    $F$-isomorphic.

\item\label{thm:ordi.c} The Clifford parallelisms\/ $\lep$ and\/ $\rip$ are
    the only Clifford-like parallelisms of the projective double
    space\/ $\dspH$.
\end{enumerate}
\end{thm}

\begin{proof}
\eqref{thm:ordi.a} $\Leftrightarrow$ \eqref{thm:ordi.b}. This was established
in \cite[Thm.~1 and Lemma~1]{fein+s-76a} (but note that the definition of
Pythagorean field used there is slightly different from ours). See also \cite[Thm. 9.1]{blunck+k+s+s-17z} for a proof in a more general situation.
\par
\eqref{thm:ordi.b} $\Leftrightarrow$ \eqref{thm:ordi.c}. From
Theorem~\ref{thm:equi}, all maximal subfields of $H$ are $F$-isomorphic if, and
only if, for all $L\in\cA(H)$ we have $\cA(H)=\bigcup_{c\in H^*} c^{-1}Lc$. The
last equation holds precisely when there are only two possibilities for the set
$\cF$ appearing in \eqref{eq:F}, namely either $\cF=\cA(H)$ or $\cF=\emptyset$.
This in turn is equivalent, by Theorem~\ref{thm:Clike}, to saying that $\lep$
and $\rip$ are the only Clifford-like parallelisms of $\dspH$.
\end{proof}

We continue by giving some explicit examples of Clifford-like parallelisms
using the construction from Theorem~\ref{thm:Clike}~\eqref{thm:Clike.a}.

\begin{exa}
Let $\Char H=2$. We define
\begin{equation*}
    \cD:=\{L\in\cAH\mid L \mbox{~is a separable extension of~} F \} .
\end{equation*}
The set $\cAH\setminus\cF$ comprises precisely the inseparable quadratic
extensions of $F$ that are contained in $H$. We get $\cF=\cD$, since the group
of inner automorphisms of $H$, in its natural action on $\cAH$, leaves both
$\cD$ and $\cAH\setminus\cD$ invariant. Both $\cF$ and $\cAH\setminus \cF$ are
non-empty; see, among others, \cite[pp.~103--104]{draxl-83} or
\cite[pp.~46--48]{tits+w-02a}. So $\cD$ gives rise to a Clifford-like
parallelism of $\dspH$ other than $\lep$ and $\rip$.
\end{exa}

\begin{exa}[{see \cite[4.12]{blunck+p+p-10a}}]
Let $H$ be the ordinary quaternion skew field over the field $\bQ$ of
rational numbers. Then each quadratic field extension $\bQ(\sqrt{-q})$ with
$q\in\bQ^*$ sum of three squares appears as a subfield of $H$.
Any two such extensions $\bQ(\sqrt{-q_1})$ and $\bQ(\sqrt{-q_2})$ are
$\bQ$-isomorphic if, and only if, $q_1$ and $q_2$ are in the same square
class of $\bQ^*$, \emph{i.e.}, there exists $c\in\bQ^*$ such that
$q_1=c^2q_2$. Since we have many non $\bQ$-isomorphic quadratic
extensions of $\bQ$ contained in $H$, we also have many possible choices for the set $\cF$ and
consequently many different Clifford-like parallelisms.
\end{exa}

Take notice that Clifford-like parallelisms of $\dspH$ always come in pairs. We
just have to change the roles of $\cF$ and $\cAH\setminus\cF$ in
\eqref{eq:Pi_F}. However, with two obvious exceptions, the two parallelisms of
such a pair do not make $\bPH$ into a projective double space. This follows
from our final theorem, which contains an even stronger result.
\begin{thm}\label{thm:really_new}
Let\/ $\parallel$ be a Clifford-like parallelism of $\dspH$ other than\/ $\lep$
and\/ $\rip$. Then there is no parallelism\/ $\parallel'$ on $\bPH$ that makes
$\bPH$ into a projective double space\/
$\bigl(\bPH,{\parallel},{\parallel'}\bigr)$.
\end{thm}

\begin{proof}
We assume, to the contrary, that there is a projective double space
$\bigl(\bPH,{\parallel},{\parallel'}\bigr)$. Also, for all $M,N\in\cLH$ let $\cR_\ell(M,N)$ denote the set of all
lines in $\leS(M)$ that have at least one common point with $N$. The sets
$\cR_r(M,N)$, $\cR(M,N)$, and $\cR'(M,N)$ are defined in an analogous way by
replacing $\lep$ with $\rip$, $\parallel$, and $\parallel'$, respectively.
\par
We claim that there exist three distinct lines $L_1$, $L_2$, $L_3$ through the
point $F1$ such that
\begin{equation}\label{eq:M123}
   \leS(L_1)=\cS(L_1),~~ \leS(L_2)=\cS(L_2),~~ \riS(L_3)=\cS(L_3) .
\end{equation}
Indeed, the existence of $L_1$ and $L_3$ is clear from $\parallel$ being
different from $\lep$ and $\rip$. By Corollary~\ref{cor:L_fixed}, $L_2$
can be chosen as $L_2:=c^{-1}L_1c$ with $c\in H^*\setminus(L_1\cup L_1^\perp)$.
We distinguish two cases.
\par
\emph{Case} (i). The parallelisms ${\parallel}$ and ${\parallel'}$ coincide. In
$\bigl(\bPH,{\parallel},{\parallel}\bigr)$ the double space axiom holds. This
gives that each line of $\cR(L_1,L_2)$ has a point in common with each line of
$\cR(L_2,L_1)$. Since the lines of $\cR(L_1,L_2)$ are mutually skew, we obtain
that $\cR(L_1,L_2)$ and $\cR(L_2,L_1)$ are mutually opposite reguli. The same
kind of reasoning in $\dspH$ gives that $\cR_\ell(L_1,L_2)$ and
$\cR_r(L_2,L_1)$ are mutually opposite reguli. From the first equation in
\eqref{eq:M123}, $\cR_\ell(L_1,L_2)=\cR(L_1,L_2)$ and, since a regulus has a
unique opposite regulus, $\cR_r(L_2,L_1)=\cR(L_2,L_1)$. The second equation in
\eqref{eq:M123} gives $\cR_r(L_2,L_1)\subseteq\leS(L_2)$. By
Theorem~\ref{thm:equi}, $\leS(L_2)\cap\riS(L_2)$ contains at most two lines,
whereas the cardinality of the regulus
$\cR_r(L_2,L_1)\subseteq\leS(L_2)\cap\riS(L_2)$ is $|F|+1$. So, this case is
impossible.

\emph{Case} (ii). The parallelisms ${\parallel}$ and ${\parallel'}$ are different.
We proceed like before and obtain in a first step that
$\cR_\ell(L_1,L_3)=\cR(L_1,L_3)$ and $\cR_r(L_3,L_1)=\cR'(L_3,L_1)$ are
mutually opposite reguli. The third equation in \eqref{eq:M123} gives
$\cR'(L_3,L_1)\subseteq\cS(L_3)$. Taking into account that
$\bigl(\bPH,{\parallel},{\parallel'}\bigr)$ admits a description in terms of
some quaternion skew field, we apply Theorem~\ref{thm:equi} and get
$|\cS(L_3)\cap\cS'(L_3)|\leq 2$, whereas
$\cR'(L_3,L_1)\subseteq\cS(L_3)\cap\cS'(L_3)$ has $|F|+1$ elements, a
contradiction.
\end{proof}

As a consequence of Theorems \ref{thm:ordi} and \ref{thm:really_new}, we obtain
the following:
\begin{cor}
A projective double space $\dspH$, where $H$ does not satisfy condition
\eqref{thm:ordi.a} from Theorem \ref{thm:ordi}, admits Clifford-like
parallelisms that are not Clifford w.r.t.\ any double space structure on $\bP$.
\end{cor}

\noindent
Hans Havlicek\\
Institut f\"{u}r Diskrete Mathematik und Geometrie\\
Technische Universit\"{a}t\\
Wiedner Hauptstra{\ss}e 8--10/104\\
A-1040 Wien\\
Austria\\
\texttt{havlicek@geometrie.tuwien.ac.at}
\par~\par
\noindent Stefano Pasotti\\
DICATAM-Sez. Matematica\\
Universit\`{a} degli Studi di Brescia\\
via Branze, 43\\
I-25123 Brescia\\
Italy\\
\texttt{stefano.pasotti@unibs.it}
\par~\par
\noindent
Silvia Pianta\\
Dipartimento di Matematica e Fisica\\
Universit\`{a} Cattolica del Sacro Cuore\\
via Trieste, 17\\
I-25121 Brescia\\
Italy\\
\texttt{silvia.pianta@unicatt.it}

\end{document}